\documentclass[article]{amsart}
\usepackage{amssymb,amsfonts,amsmath,amsthm}
\usepackage[all,arc]{xy}
\usepackage{enumerate}
\usepackage{mathrsfs}
\usepackage[toc,page]{appendix}
\usepackage[font=small,labelfont=bf]{caption}
\usepackage[left=3cm, right=3cm, bottom=3cm]{geometry}
\usepackage{tabularx}
\usepackage{url}
\usepackage{color}
\usepackage{tikz-cd}

\newtheorem{thm}{Theorem}[section]

\newtheorem*{thm*}{Theorem}
\newtheorem{cor}[thm]{Corollary}
\newtheorem{prop}[thm]{Proposition}
\newtheorem*{prop*}{Proposition}

\newtheorem*{lem*}{Lemma}

\theoremstyle{definition}

\newtheorem{exmp}[thm]{Example}

\theoremstyle{remark}
\newtheorem{rem}[thm]{Remark}

\newtheorem*{idea*}{Idea}

\newcommand{\Spec}{\text{Spec }}

\makeatletter
\let\c@equation\c@thm
\makeatother
\numberwithin{thm}{section}
\numberwithin{equation}{section}

\title{Deformation of Locally Free Sheaves and Hitchin Pairs over Nodal curve}

\author{Hao Sun}

\begin{document}
\maketitle
\begin{abstract}
In this article, we study the deformation theory of locally free sheaves and Hitchin pairs over a nodal curve. As a special case, the infinitesimal deformation of these objects gives the tangent space of the corresponding moduli spaces, which can be used to calculate the dimension of the corresponding moduli space. We show that the deformation of locally free sheaves and Hitchin pairs over a nodal curve is equivalent to the deformation of generalized parabolic bundles and generalized parabolic Hitchin pairs over the normalization of the nodal curve respectively.
\end{abstract}
\flushbottom

\section{Introduction}
The moduli space of semistable locally free sheaves (coherent sheaves) and Hitchin pairs over a smooth curve is studied by many mathematicians and is by now well-understood. The moduli space of Hitchin pairs over a smooth curve was first constructed by Hitchin in \cite{Hit} and generalized by Nitsure in \cite{Nit}. Later on, Biswas and Ramanan \cite{BisRam} studied the infinitesimal deformation of Hitchin pairs. This deformation theory provides a way to study the tangent space of the moduli space of Hitchin pairs and the dimension of this moduli space.

In the last several decades, attention began to focus on the locally free sheaves and Hitchin pairs over the nodal curve. Bhosle has showed in \cite{Bho1} that there is a correspondence between bundles over a nodal curve and generalized parabolic bundles over its normalization \cite{Bho1}. Later on, Bhosle showed that this correspondence can be extended to Hitchin pairs, more precisely, between Hitchin pairs over nodal curve and generalized parabolic Hitchin pairs over its normalization \cite{Bho2}. Under this correspondence, studying the deformation theory of Hitchin pairs over a nodal curve is equivalent to studying the deformation theory of the corresponding generalized parabolic Hitchin pairs over its normalization.

In this article, after providing the necessary background in \S 2, in \S 3, we study the deformation theory of locally free sheaves over a nodal curve $X$. Let $C'$, $C$ be two local Artin rings over a field $k$ satisfying the following exact sequence
\begin{align*}
0 \longrightarrow J \longrightarrow C' \longrightarrow C \longrightarrow 0,
\end{align*}
where $J$ is an ideal such that $\mathfrak{m}_{C'} J =0$. Let $X$ be a nodal curve over $C$ and let $X'$ be an extension of $X$ flat over $C'$. Equivalently, $X' \times_{\Spec C'} \Spec C=X$. We fix a locally free sheaf $\mathcal{E}$ over $X$. We say that a locally free sheaf $\mathcal{E}'$ over $X'$ is a \emph{deformation} of $\mathcal{E}$, if $\mathcal{E}' \otimes_{\mathcal{O}_{X'}}\mathcal{O}_{X} \cong \mathcal{E}$. We review the necessary definitions for the deformation theory and pseudotorsor in \S 3 and the reader can also find those in \cite[Chap 6]{Hart}.
\begin{thm*}{\textbf{\ref{301}}}
let $\mathcal{E}$ be a locally free sheaf over the nodal curve $X$.
\begin{itemize}
\item[(1)] The set of deformations $\mathcal{E}'$ over $X'$ is a pseudotorsor under the action of the additive group $H^0(X, \mathcal{E}^* \otimes J\otimes_C R_\mathcal{E})$.
\item[(2)] If the extensions of $\mathcal{E}'$ over $X'$ exist locally on $X$, then there is an obstruction $\phi \in H^1(X, \mathcal{E}^* \otimes J\otimes_C R_{\mathcal{E}})$, whose vanishing is necessary and sufficient for the global existence of $E'$. If such a deformation $\mathcal{E}'$ over $\mathcal{E}$ exists, then the set of all such deformations is a torsor under $H^0(X, \mathcal{E}^* \otimes J\otimes_C R_\mathcal{E})$.
\end{itemize}
\end{thm*}

In \S 4, we study the deformation theory of $\mathbb{L}$-twisted Hitchin pairs $(E,\Phi)$ over a nodal curve $X$, where $\mathbb{L}$ is a fixed line bundle over $X$. Let $\rho$ be the natural action of $\text{End}(E)$ on itself. The deformation complex $C_{J}^{\bullet}$ is defined as follows
\begin{align*}
C_{J}^{\bullet}: C_{J}^0 = \text{End}(E) \otimes J \xrightarrow{e(\Phi)} C_J^1 = \text{End}(E) \otimes \mathbb{L} \otimes J,
\end{align*}
where the map $e(\Phi)$ is given by
\begin{align*}
e(\Phi)(s)=-\rho(s)(\Phi).
\end{align*}
lWe generalize the proof of \cite[Theorem 2.3]{BisRam} and have the following theorem.
\begin{thm*}{\textbf{\ref{401}}}
The set of deformations of $(E,\Phi)$ is isomorphic to $\mathbb{H}^1(C_J^{\bullet})$, where $C_J^{\bullet}$ is the complex $$C_J^{\bullet}:C_J^0 =End(E) \otimes J \xrightarrow{e(\Phi)} C_J^1 = \text{End}(E) \otimes \mathbb{L} \otimes J.$$
\end{thm*}
Under the correspondence between $\mathbb{L}$-twisted Hitchin pairs over the nodal curve $X$ and $\widetilde{L}$-twisted generalized parabolic Hitchin pairs over its normalization $\widetilde{X}$ \cite{Bho2}, the above theorem can be interpreted by the deformation of generalized parabolic Hitchin pairs over $\widetilde{X}$. We discuss this property in details in Corollary \ref{403} and Remark \ref{404}.

\section{Background}
\subsection{Principal $\mathbb{L}$-Twisted Higgs Bundles over a Nodal Curve}
Let $X$ be an irreducible nodal curve over $\mathbb{C}$ and $\widetilde{X}$ the normalization of $X$. Denote by $\nu : \widetilde{X} \rightarrow X$ the normalization map. For convenience, we assume that there is only one simple node $x_0$ in $X$, and $\widetilde{x}_1,\widetilde{x}_2 \in \nu^{-1}(x_0)$ are preimages of $x_0$ under $\nu$. Denote by $\mathcal{M}(X,n,d)$ the moduli space of isomorphism classes of semistable $\text{GL}(n,\mathbb{C})$-bundles $E$ with rank $n$, degree $d$ over the nodal curve $X$. The construction of this moduli space is well-known (see for instance \cite{HuyLeh}).

Now we fix a line bundle $\mathbb{L}$ over $X$. Let $G$ be a linear algebraic group. A \emph{principal $\mathbb{L}$-twisted Higgs $G$-bundle} over $X$ is a pair $(E,\Phi)$ consisting of a principal $G$-bundle $E$ over $X$ and a section $\Phi : X \rightarrow \text{ad}(E) \otimes \mathbb{L}$, where $\text{ad}(E)=E \times_{\text{Ad}} \mathfrak{g}$ is the adjoint representation of $E$ and $\mathfrak{g}$ is the Lie algebra of $G$. Let $\rho: G \rightarrow \text{GL}(V)$ be a faithful representation. We say that the Higgs bundle $(E,\Phi)$ is \emph{stable} (resp. \emph{semistable}), if for any $\Phi$-invariant sub-bundle $F$, we have
\begin{align*}
\frac{\deg F}{\text{rk}F} < \frac{\deg E}{\text{rk}E}, \text{ (resp. $\leq$ )}.
\end{align*}
Let $\mathcal{M}(X,G,\mathbb{L})$ be the moduli space of isomorphism classes of semistable $\mathbb{L}$-twisted principal $G$-Higgs bundle $(E,\Phi)$ over the nodal curve $X$. The moduli space $\mathcal{M}(X,G,\mathbb{L})$ co-represents the moduli problem
\begin{align*}
\mathcal{M}(X,G,\mathbb{L})_{func}:\text{Sch}_{\mathbb{C}} \rightarrow \text{ Sets },
\end{align*}
and given $S \in \text{Sch}_{\mathbb{C}}$, $\mathcal{M}(X,n,\mathbb{L})_{func}(S)$ is the set of isomorphism classes of families of $\mathbb{L}$-twisted semistable principal $G$-Higgs bundles over the nodal curve $X$ parametrized by $S$ \cite{GiuPus}. The authors in \cite{GiuPus} also proved that this moduli space $\mathcal{M}(X,G)$ is a projective scheme.
\begin{thm}[Theorem 1 in \cite{GiuPus}]\label{201}
The moduli space $\mathcal{M}(X,G,\mathbb{L})$ is a projective scheme which co-represents the moduli problem $\mathcal{M}(X,G,\mathbb{L})_{func}$.
\end{thm}

In this paper, we focus on the case $G=\text{GL}(n,\mathbb{C})$, and consider the associated bundle $E \times_{\rho} V$ and the associated Higgs field $\Phi$. We use the same notation $(E,\Phi)$ for the associated Higgs bundle. Denote by $\mathcal{M}(X,n,d,\mathbb{L})$ the moduli space of isomorphism classes of semistable $\mathbb{L}$-twisted $\text{GL}(n,\mathbb{C})$-Higgs bundles $(E,\Phi)$ with rank $n$, degree $d$ over the nodal curve $X$.

\subsection{Generalized Parabolic Hitchin Pairs}
We review the definition and some properties of the generalized parabolic Hitchin pair in this subsection. Details can be found in \cite{Bho1,Bho2}.

Let $Y$ be an irreducible non-singular algebraic curve defined over an algebraically closed field $k$. Let $\mathbb{L}_Y$ be a fixed line bundle over $Y$. We fix $s$-many disjoint Cartier divisors $D_i$ on $Y$, $1 \leq i \leq s$. Let $D=\sum\limits_{i=1}^s D_i$. Let $E$ be a locally free sheaf over $Y$. Denote by $n$ and $d$ the rank and degree of $E$. In this paper, a holomorphic bundle over $Y$ is exactly a locally free sheaf. Sometimes, we abuse the terminology.

A \emph{generalized parabolic $\mathbb{L}_Y$-twisted Hitchin pair} (GPH) of rank $n$ and degree $d$ on $(Y,D)$ is a tuple $(E,F(E),\alpha,\Phi)$, where
\begin{enumerate}
\item[(1)] $E$ is a locally free sheaf on $Y$ with rank $n$ and degree $d$,
\item[(2)] $F(E)=(F_1(E),\dots,F_s(E))$ is a $s$-tuple such that $F_i(E) \subseteq E \otimes \mathcal{O}_{D_i}$,
\item[(3)] $\alpha=(\alpha_1,\dots,\alpha_s)$ is a $s$-tuple of rational numbers $0 \leq \alpha_i < 1$, $1 \leq i \leq s$, and $\alpha_i$ is known as the weight of the filtration given in condition (2),
\item[(4)] $\Phi : E \rightarrow E \otimes \mathbb{L}_Y$ is a homomorphism preserving the filtration, i.e. $\Phi(F_i(E)) \subseteq F_i(E) \otimes \mathbb{L}_Y$.
\end{enumerate}
Condition $(4)$ is exactly the definition of homomorphism between (generalized) parabolic bundles \cite{Yoko2}. Compared with the homomorphisms of holomorphic bundles, parabolic homomorphisms need to preserve the parabolic structures. More precisely, let $\text{ParEnd}(E)$ be the set of parabolic homomorphisms of the generalized parabolic bundle $E$ and let $E_{D_i}=E \otimes \mathcal{O}_{D_i}$. Define $P_{D_i}(E,E)$ to be the subspace of $\text{End}(E_{D_i},E_{D_i})$ consisting of maps preserving the filtration over $D_i$. We have
\begin{align*}
0 \rightarrow \text{ParEnd}(E) \rightarrow \text{End}(E) \rightarrow \text{End}(E_D,E_D)/ P_D(E,E) \rightarrow 0,
\end{align*}
where $E_D= \bigoplus\limits_{i=1}^s E_{D_i}$ and $P_D(E,E)=\bigoplus\limits_{i=1}^s P_{D_i}(E,E)$.

A \emph{generalized parabolic bundle} is a pair $(E,F(E))$ satisfying conditions (1) and (2) \cite{Bho1}. In this paper, we assume that all weights $\alpha_i$ are the same. Denote by $\alpha:=\alpha_1=\dots=\alpha_s$. Usually, we consider $GPH$ as a triple $(E,F(E),\Phi)$ with a given weight $\alpha$. Let $f_i(E)= \dim F_i(E)$ be the dimension of the filtration. We define the \emph{weight} $wt(E)$, and \emph{parabolic degree} $\text{par}\deg(E)$ of the locally free sheaf $E$ as follows
\begin{align*}
wt(E)=\alpha  \sum_{i=1}^s f_i(E), \quad \text{par}\deg(E)=d(E)+wt(E).
\end{align*}
The parabolic slope $\text{par}\mu$ is defined by
\begin{align*}
\text{par}\mu(E)=\frac{\text{par}\deg(E)}{r_{E}}.
\end{align*}
A parabolic bundle $E'$ is a \emph{parabolic subbundle} of $E$, if $E'$ is a subbundle of $E$, and its filtration $F_i(E')$ satisfies $F_i(E')=F_i(E) \cap (E' \otimes \mathcal{O}_{D_j})$ such that the weights $\alpha_i(E')=\alpha_i(E)$. It is called a \emph{$\Phi$-invariant subbundle}, if $\Phi(E') \subseteq E' \otimes \mathbb{L}_Y$.

A generalized parabolic bundle $(E,F(E))$ is called \emph{stable (resp. semistable)}, if for every proper subbundle $E' \subseteq E$, we have
\begin{align*}
\text{par}\mu(E') < \text{par}\mu(E), \quad \text{ (resp. $\leq$) }.
\end{align*}
Denote by $\mathcal{M}_{par}(Y,n,d)$ the moduli space of isomorphism classes of semistable generalized parabolic bundles $(E,F(E))$ with rank $n$, parabolic degree $d$ over the smooth curve $Y$. The existence of the moduli space $\mathcal{M}_{par}(Y,n,d)$ is given in \cite{Bho1} Theorem 1 and Theorem 3.

A GPH $(E,F(E),\Phi)$ is \emph{$\alpha$-stable (resp. $\alpha$-semistable)}, if for every proper $\Phi$-invariant subbundle $E' \subseteq E$, we have
\begin{align*}
\text{par}\mu(E') < \text{par}\mu(E), \quad \text{ (resp. $\leq$) }.
\end{align*}
A GPH $(E,F(E),\Phi)$ is \emph{stable (resp. semistable)}, if it is $1$-stable (resp. $1$-semistable). Denote by $\mathcal{M}_{par}(Y,n,d,\alpha,\mathbb{L}_Y)$ the moduli space of isomorphism classes of $\alpha$-semistable $\mathbb{L}_Y$-twisted generalized parabolic Hitchin pairs (GPH) $(E,F(E),\Phi)$ with rank $n$, parabolic degree $d$ over the smooth curve $Y$. Moreover, denote by $\mathcal{M}_{par}(Y,n,d,\mathbb{L}_Y)$ the moduli space of semistable $\mathbb{L}_Y$-twisted generalized parabolic Hitchin pairs (GPH) with the same condition as above.

The existence of the moduli space $\mathcal{M}_{par}(Y,n,d,\alpha,\mathbb{L}_Y)$ of GPH is given by Bhosle \cite{Bho2}.
\begin{thm}[Theorem 4.8 in \cite{Bho2}]\label{202}
Let $Y$ be a smooth algebraic curve of genus $g$. We fix a line bundle $\mathbb{L}_Y$ and a rational number $\alpha$, $0 < \alpha \leq 1$. Let $D_j$, $1 \leq j \leq s$, be finitely many distinct divisors on $X$. There exists a moduli space $\mathcal{M}_{par}(Y,n,d,\alpha,\mathbb{L}_Y)$ of isomorphism classes of $\alpha$-semistable $\mathbb{L}_Y$-twsited GPH $(E,F(E),\Phi)$, where $E$ is a holomorphic bundle of rank $n$, degree $d$ with following filtration
\begin{align*}
E \otimes \mathcal{O}_{D_j} \supset F_j(E) \supset 0,
\end{align*}
and associated weights $(0,\alpha)$, $1 \leq j \leq s$, and $\Phi: E \rightarrow E \otimes \mathbb{L}_Y$ is a homomorphism of parabolic bundles. The moduli space $\mathcal{M}_{par}(Y,n,d,\mathbb{L}_Y)$ is a projective scheme.
\end{thm}

It is well known that studying generalized parabolic Hitchin pair is closely related to the study of the Hitchin pair over the nodal curve \cite{Bho2}. Here is a brief review of this relation. Let $X$ be an integral projective nodal curve and $\widetilde{X}$ its normalization. Let $\nu :\widetilde{X} \rightarrow X$ be the normalization map. Let $x_1,\dots,x_s$ be the nodes of $X$. We define the divisor $D_i \subseteq \widetilde{X}$ as the preimage of $x_i$. Clearly, $D_i$ is the sum of two points. Let $\widetilde{\mathcal{O}}_{x_i}$ be the normalization of the local ring $\mathcal{O}_{x_i}$ at $x_i$. In this case, it is easy to check that $\dim(\widetilde{\mathcal{O}}_{x_i}/ \mathcal{O}_{x_i})=1$. Given a line bundle $\mathbb{L}$ on $X,$ define $\widetilde{\mathbb{L}}=\nu^* \mathbb{L}$. Let $(\widetilde{E},F(\widetilde{E}),\Phi_{\widetilde{E}})$ be a GPH with weight $\alpha$ over $(\widetilde{X},D=\sum\limits_{i=1}^s D_i)$. We take
\begin{align*}
f_j(\widetilde{E})=r(\widetilde{E}).
\end{align*}
A GPH $(\widetilde{E},F(\widetilde{E}),\Phi_{\widetilde{E}})$ is \emph{good}, if it satisfies the following conditions:
\begin{itemize}
\item[(1)] the space $F_i(\widetilde{E})$ is an $\mathcal{O}_{x_i}$-sub-module of $\nu_*(\widetilde{E} \otimes \mathcal{O}_{D_j})$,
\item[(2)] we have $\nu_*(\widetilde{\Phi}) (\nu_*(F_i(\widetilde{E}))) \subseteq \nu_*(F_i(\widetilde{E})) \otimes \mathbb{L}_{x_i}$, $1 \leq i \leq s$.
\end{itemize}
Note that the definition of good GPH does not depend on the weight $\alpha$. The good GPH forms a closed subscheme $\mathcal{M}_{par}^{good}(\widetilde{X},n,d,\widetilde{\mathbb{L}})$ of $\mathcal{M}_{par}(\widetilde{X},n,d,\widetilde{\mathbb{L}})$.

There is a one-to-one correspondence between good $\widetilde{\mathbb{L}}$-twisted GPHs over $\widetilde{X}$ and $\mathbb{L}$-twisted Hitchin pairs over the nodal curve $X$.
\begin{prop}[Proposition 2.8 in \cite{Bho2}]\label{203}
We forget the weight $\alpha$ in this proposition.
\begin{enumerate}
\item[(1)] A $\widetilde{\mathbb{L}}$-twisted good GPH $(\widetilde{E},F(\widetilde{E}),\Phi_{\widetilde{E}})$ of rank $n$, degree $d$ on $\widetilde{X}$ defines a $\mathbb{L}$-twisted Hitchin pair $(E,\Phi)$ of rank $n$ and degree $d$ on $X$.
\item[(2)] If $(E,\Phi)$ is a $\mathbb{L}$-twisted Hitchin pair on $X$, then $(E,\Phi)$ determines a $\widetilde{\mathbb{L}}$-twisted good GPH $(\widetilde{E},F(\widetilde{E}),\Phi_{\widetilde{E}})$ on $\widetilde{X}$.
\end{enumerate}
\end{prop}
This correspondence induces a birational morphism between $\mathcal{M}_{par}^{good}(\widetilde{X},n,d,\widetilde{\mathbb{L}})$ and $\mathcal{M}(X,n,d,\mathbb{L})$.
\begin{thm}[Theorem 1.2 in \cite{Bho2}]\label{204}
There exists a birational morphism
\begin{align*}
\mathcal{M}_{par}^{good}(\widetilde{X},n,d,\widetilde{\mathbb{L}}) \rightarrow \mathcal{M}(X,n,d,\mathbb{L})
\end{align*}
from the moduli space of $\widetilde{\mathbb{L}}$-twisted semistable good GPH on $\widetilde{X}$ to the moduli space of semistable $\mathbb{L}$-twisted Hitchin pairs on $X$.
\end{thm}
Bhosle discussed in detail the relation between $\alpha$-semistable $\widetilde{\mathbb{L}}$-twisted good GPH on $\widetilde{X}$ and semistable $\widetilde{L}$-twisted Hitchin pair on $X$ for any $\alpha \in (0,1]$. In this paper, we shall only consider the case $\alpha=1$.

\subsection{Infinitesimal Deformation of Hitchin Pair over Nonsingular Algebraic Curve}
Let $(E,\Phi)$ be a $\mathbb{L}_Y$-twisted Higgs bundle over a nonsingular projective curve $Y$. An \emph{infinitesimal deformation} of the Hitchin pair $(E,\Phi)$ is a pair $(E',\Phi')$ over $Y \times \Spec \mathbb{C}[\varepsilon]/(\varepsilon^2)$ with an isomorphism of the restriction to $\mathfrak{m} \times Y$, where $\mathfrak{m}$ is the closed point of $\mathbb{C}[\varepsilon]/(\varepsilon^2)$. Now we consider the $\mathbb{L}_Y$-twisted Higgs bundle $E[\varepsilon]=E \times \Spec \mathbb{C}[\varepsilon]/(\varepsilon^2)$. The automorphisms of $E[\varepsilon]$ which induce identity over the closed point is $\text{End}(E)$. Therefore for a section $s$ of $\text{End}(E)$, the corresponding automorphism of $E[\varepsilon]$ is denoted by $1+s\varepsilon$. Moreover, if $v+w\varepsilon$ is a section of $(\text{End}(E) \otimes \mathbb{L}_Y)[\varepsilon]$, we have
\begin{align*}
\rho(1+s\varepsilon)(v+w\varepsilon)=v+w\varepsilon+\rho(s)(v)\varepsilon,
\end{align*}
where $\rho$ is natural action of $\text{End}(E)$ on $E$. The deformation complex $C^{\bullet}$ is defined as follows
\begin{align*}
C^{\bullet}: C^0 = \text{End}(E) \xrightarrow{e(\Phi)} C^1 = \text{End}(E) \otimes \mathbb{L}_Y,
\end{align*}
where the map $e(\Phi)$ is given by
\begin{align*}
e(\Phi)(s)=-\rho(s)(\Phi).
\end{align*}
The authors in \cite{BisRam} used this complex to calculate the space of infinitesimal deformations of the Hitchin pair $(E,\Phi)$ over a nonsingular algebraic curve $Y$.
\begin{thm}[Theorem 2.3 in \cite{BisRam}]\label{205}
The space of infinitesimal deformations of a given $\mathbb{L}_Y$-twisted Hitchin pair $(E,\Phi)$ over $Y$ is isomorphic to the first hypercohomology group $\mathbb{H}^1(C^{\bullet})$ of the complex $C^{\bullet}$.
\end{thm}

\section{Deformation of Locally Free Sheaves over Nodal Curve}
In this section, we want to study the (infinitesimal) deformation theory of locally free sheaves over a nodal curve $X$, which will give us a way to calculate the tangent space of $\mathcal{M}(X,n,d)$. We first review the definition of deformation theory from \cite{Hart}.

Let $C'$, $C$ be two local Artin rings over a field $k$ with maximal ideals $\mathfrak{m}_{C'}$, $\mathfrak{m}_{C}$ respectively satisfying the following exact sequence
\begin{equation}\label{exsq3}
0 \longrightarrow J \longrightarrow C' \longrightarrow C \longrightarrow 0,
\end{equation}
where $J$ is an ideal such that $\mathfrak{m}_{C'} J =0$. Thus we can consider $J$ as a $k$-vector space, where $k$ is the residue field of $C$.

Let $X$ be a scheme over $C$ and let $X'$ be an extension of $X$ flat over $C'$. In other words, $X'$ is a flat family over $\Spec C'$ and there is a closed embedding $X \hookrightarrow X'$ such that $X' \times_{\Spec C'} \Spec C=X$. We fix a locally free sheaf $\mathcal{E}$ over $X$. In this section, we will consider the deformation problems over the sequence (\ref{exsq3}). We say that a locally free sheaf $\mathcal{E}'$ over $X'$ is a \emph{deformation} of $\mathcal{E}$, if $\mathcal{E}' \otimes_{\mathcal{O}_{X'}}\mathcal{O}_{X} \cong \mathcal{E}$. If we work on the following exact sequence
\begin{equation}\label{exsq4}
0 \longrightarrow (\varepsilon) \cong k \longrightarrow k[\varepsilon]/(\varepsilon^2) \longrightarrow k \longrightarrow 0,
\end{equation}
where $k$ is a field with characteristic $0$, we say that $\mathcal{E}'$ is an \emph{infinitesimal deformation} of $\mathcal{E}$.

In this paper, we study the deformation theory of locally free sheaves overl the nodal curve $X$. Let $\widetilde{X}$ be the normalization of $X$. Denote by $\pi: \widetilde{X} \rightarrow X$ the natural projection map. We first work on this problem in the affine case. Let $X=\Spec A$ be an affine space over $\Spec C$ and $\widetilde{X}= \Spec \widetilde{A}$ its normalization. We have a short exact sequence
\begin{equation}\label{exsq1}
0 \longrightarrow A \longrightarrow \widetilde{A} \longrightarrow R \longrightarrow 0,
\end{equation}
where $R$ is an $A$-module. Let $E$ be a fixed $A$-module. We have the following exact sequence
\begin{equation}\label{exsq2}
0 \longrightarrow E \longrightarrow \pi_*\widetilde{E} \longrightarrow R_E \longrightarrow 0,
\end{equation}
where $\widetilde{E}= \pi^* E = E \otimes_A \widetilde{A}$ and $R_E=E \otimes_A R$. Note that $\widetilde{E}$ is exactly the bundle corresponding to $E$ in Proposition \ref{203}. The parabolic structure comes from $R_E$. More precisely, we have
\begin{align*}
0 \longrightarrow E \longrightarrow \pi_*\widetilde{E} \longrightarrow \pi_*\left( \sum_{i} \frac{\widetilde{E} \otimes \mathcal{O}_{D_i}}{F_j(E)} \right) \longrightarrow 0,
\end{align*}
where the sum is over all nodes $x_i$ of $X$ and $D_i$ is the preimage of the node $x_i$ in $\widetilde{X}$. Details about this exact sequence can be found in \cite{Bho1,Bho2}. The term $R_E$ in Eq \eqref{exsq2} is exactly $\pi_*\left( \sum_{i} \frac{\widetilde{E} \otimes \mathcal{O}_{D_i}}{F_j(E)} \right)$.

We fix an extension $X'=\Spec A'$ of $X$. Exact sequences \eqref{exsq1} and \eqref{exsq3} then provide the following $3\times3$ commutative diagram,
\begin{center}
\begin{tikzcd}	
& 0 \arrow[d] & 0 \arrow[d] & 0 \arrow[d] &\\
	0 \arrow[r] & J \otimes_C A \arrow[r] \arrow[d] & A' \arrow[r] \arrow[d] & A \arrow[r] \arrow[d] & 0\\

  	0 \arrow[r] & J \otimes_C \widetilde{A} \arrow[r] \arrow[d] & \widetilde{A}' \arrow[r] \arrow[d] & \widetilde{A} \arrow[r] \arrow[d]& 0\\
	
   	0 \arrow[r] & J \otimes_C R \arrow[r] \arrow[d] & R' \arrow[r] \arrow[d] & R \arrow[r] \arrow[d]& 0\\
      & 0 & 0 & 0 &
\end{tikzcd}	
\end{center}
where $\widetilde{A}'=\widetilde{A} \otimes_{A} A'$. Given an $A$-module $E$, let $\widetilde{E}':=E \otimes_A \widetilde{A}'$. We want to classify deformations $E'$ of $E$ over $A'$. In other words, we want to find all $A'$-modules $E'$ such that $\widetilde{E}' \otimes_{\widetilde{A}'} A' =E'$ and satisfying the following $3\times3$ commutative diagram.
\begin{center}
\begin{tikzcd}	
& 0 \arrow[d] & 0 \arrow[d] & 0 \arrow[d] &\\
	0 \arrow[r] & J \otimes_C E \arrow[r] \arrow[d] & E' \arrow[r] \arrow[d] & E \arrow[r] \arrow[d] & 0\\

 0 \arrow[r] & J \otimes_C \pi_* \widetilde{E} \arrow[r] \arrow[d] & \pi_* \widetilde{E}' \arrow[r] \arrow[d] & \pi_* \widetilde{E} \arrow[r] \arrow[d]& 0\\
	
    	0 \arrow[r] & J \otimes_C R_E \arrow[r] \arrow[d] & R'_E \arrow[r] \arrow[d] & R_E \arrow[r] \arrow[d]& 0\\
      & 0 & 0 & 0 &
\end{tikzcd}	
\end{center}

Before we state the result, we want to give the definition about \emph{torsor} and \emph{pseudotorsor} \cite{Hart}. Let $G$ be a group acting on a set $S$. We say that $S$ is a \emph{torsor} under the action of $G$, if it satisfies the following two conditions:
\begin{itemize}
\item[(1)] For every $s \in S$, the induced mapping $g \mapsto g(s)$ is a bijective map from $G$ to $S$,
\item[(2)] the set $S$ is nonempty.
\end{itemize}
We say that $S$ is a \emph{pseudotorsor}, if it satisfies condition $(1)$ above.

\begin{thm}\label{301}
With the same notation as above, let $\mathcal{E}$ be a locally free sheaf over the nodal curve $X$.
\begin{itemize}
\item[(1)] The set of deformations $\mathcal{E}'$ of $\mathcal{E}$ over $X'$ is a pseudotorsor under the action of the additive group $H^0(X, \mathcal{E}^* \otimes J\otimes_C R_\mathcal{E})$.
\item[(2)] If the extension $\mathcal{E}'$ of $\mathcal{E}$ over $X'$ exist locally on $X$, then there is an obstruction $\phi \in H^1(X, \mathcal{E}^* \otimes J\otimes_C R_{\mathcal{E}})$, whose vanishing is necessary and sufficient for the global existence of $E'$. If such a deformation $\mathcal{E}'$ of $\mathcal{E}$ exists, then the set of all such deformations is a torsor under $H^0(X, \mathcal{E}^* \otimes J\otimes_C R_\mathcal{E})$.
\end{itemize}

\end{thm}

\begin{proof}
We first consider this problem in the affine case and we will use the second $3\times3$ commutative diagram for $E$. Let $E'_1$ and $E'_2$ be two possible choices for $E'$. Let $x_1 \in E'_1$ and $x_2 \in E'_2$ be two elements with the same image $x \in R_E$. Note that the choice of $x_1, x_2$ is not unique but determined only up to some element in $J \otimes_C E$. The element $x_1 - x_2$ is also a well-defined element in $J \otimes_C \widetilde{E}$, which is zero in $J \otimes_C R_E$. Thus $x \in E$ gives us a well-defined element in $J \otimes_C E$. Denote by $\varpi: E \rightarrow J \otimes_C R_E$ the map sending $x$ to the corresponding element in $J \otimes_C R_E$. It is easy to check that this map $\varpi$ is $A$-linear. Therefore we get a map $\varpi \in \text{Hom}_A(E, J\otimes_C R_E)$.

Now given $E'_1$ and a map $\varpi \in \text{Hom}_A(E, J \otimes_C R_E)$, we define another module $E'_2$ fitting into the $3\times3$ diagram. Note that $E'$ and $R'_E$ determine each other uniquely. Therefore it is equivalent to construct $(R'_E)_2$ for $E'_2$. Let $(R'_E)_2$ be the set of $x_2 \in \widetilde{E}'$, whose image $x \in \widetilde{E}$ is in $E$, such that any lifting $x_1$ of $x$ to $E'_1$, the image of $x_2 - x_1 \in J \otimes_C R_E$ is equal to $\varpi(x)$. It is easy to check that $E'_2$ is a well-defined element fitting into the diagram.

Finally, we have to check that this action is a group action. Let $E'_1,E'_2,E'_3$ be three choices of $E'$. The map $\varpi_1$ is defined by $E'_1,E'_2$, $\varpi_2$ is defined by $E'_2,E'_3$ and $\varpi_3$ is defined by $E'_1,E'_3$, then $\varpi_3=\varpi_1+\varpi_2$. Thus the operation $\varpi(E'_1)=E'_2$ is a group action with the additive group $\text{Hom}_A(E, J\otimes_C R_E)$. This additive group $\text{Hom}_A(E, J\otimes_C R_E)$ is exactly $H^0(X, \mathcal{E}^* \otimes J\otimes_C R_\mathcal{E})$. It is easy to check that if the pseudotorsor exists locally in the affine chart, it can be globalized naturally. This finishes the proof of part (1) of the theorem.

To prove (2), we assume that the deformation $\mathcal{E}'$ of $\mathcal{E}$ exists locally. In other words, there exists an open affine covering $\mathcal{X}=(X_i)_{i \in \mathcal{I}}$ of $X$, where $\mathcal{I}$ is the index set, such that on each local chart $X_i$, there exists a deformation $\mathcal{E}'_i$ of $\mathcal{E}_i=\mathcal{E}|_{X_i}$. Let $X'_i:=X_i \times_{\Spec C} \Spec C'$ be the local chart of $X'$. We first focus on the intersection $X'_{ij}=X'_i \cap X'_j$. There are two possible extensions $\mathcal{E}'_i$ and $\mathcal{E}'_j$ of $\mathcal{E}_{ij}$ on the intersection $X'_{ij}=X'_i \cap X'_j$. By part (1), these two extensions define an element $\varpi_{ij} \in H^0(X_{ij}, \mathcal{E}^* \otimes J\otimes_C R_\mathcal{E})$. On the intersection $X'_{ijk}=X'_i \cap X'_j \cap X'_k$ of three affine open sets, there are three deformations $\mathcal{E}'_i$, $\mathcal{E}'_j$ and $\mathcal{E}'_k$. The differences define the elements $\varpi_{ij}$, $\varpi_{ik}$ and $\varpi_{jk}$ in $H^0(X_{ij}, \mathcal{E}^* \otimes J\otimes_C R_\mathcal{E})$ such that $\varpi_{ik}=\varpi_{ij}+\varpi_{jk}$. Clearly, $(\varpi_{ij})$ is a $1$-cocycle for the covering $\mathcal{X}$ and the sheaf $\mathcal{E}^* \otimes J\otimes_C R_E$. If $(\mathcal{E}'^{0}_i)_{i \in \mathcal{I}}$ is another choice of local deformations. Similarly, this choice defines $\varpi^0_{ij} \in H^0(X_{ij}, \mathcal{E}^* \otimes J\otimes_C R_\mathcal{E})$ such that $(\varpi^0_{ij})$ is a $1$-cocycle. Also note that these two deformations $\mathcal{E}'_i$ and $\mathcal{E}'^0_i$ give us a well defined element $\alpha_i \in H^0(X_{i}, \mathcal{E}^* \otimes J\otimes_C R_\mathcal{E})$ such that $\alpha_i-\alpha_j=\varpi_{ij}-\varpi^0_{ij}$. Therefore the cohomology class $\alpha=(\alpha_i)$ is well-defined. This cohomology class $\alpha$ is the obstruction to the existence of a global deformation $\mathcal{E}'$ of $\mathcal{E}$ over $X'$. It is easy to check that a global deformation $\mathcal{E}'$ exists if and only if $\alpha=0$. This finishes the proof of part (2).
\end{proof}

\begin{exmp}\label{302}
In this example, we consider the infinitesimal deformation of a rank $n$, degree $0$ locally free sheaf $E$ on a nodal curve $X$ over $\mathbb{C}$ with a single node. Let $J=(\varepsilon) \cong \mathbb{C}$, $C'=\mathbb{C}[\varepsilon]/(\varepsilon^2)$ and $C=\mathbb{C}$. We use the exact sequence \eqref{exsq4}. In this case, we have
\begin{center}
\begin{tikzcd}	
& 0 \arrow[d] & 0 \arrow[d] & 0 \arrow[d] &\\
	0 \arrow[r] & E \arrow[r] \arrow[d] & E' \arrow[r] \arrow[d] & E \arrow[r] \arrow[d] & 0\\

 0 \arrow[r] & \pi_* \widetilde{E} \arrow[r] \arrow[d] & \pi_* \widetilde{E}' \arrow[r] \arrow[d] & \pi_* \widetilde{E} \arrow[r] \arrow[d]& 0\\
	
    	0 \arrow[r] & R_E \arrow[r] \arrow[d] & R'_E \arrow[r] \arrow[d] & R_E \arrow[r] \arrow[d]& 0\\
      & 0 & 0 & 0 &
\end{tikzcd}	
\end{center}
where $R_E \cong \mathbb{C}$. If $X$ has $s$ nodes, then $R_E \cong \sum\limits_{x \text{ is nodes}}\mathbb{C}_{x}$. Thus, if $H^1(X, \mathcal{E}^* \otimes J\otimes_C R_\mathcal{E})$ vanishes, we have $H^0(X, \mathcal{E}^* \otimes J\otimes_C R_\mathcal{E})=H^0(X, \mathcal{E}^*)$. It is easy to check that $\dim H^0(X, \mathcal{E}^*)=n^2(g_X-1)+1$. This number is the dimension of the tangent space of the moduli space $\mathcal{M}(X,n,0)$ at the smooth point $E$, more precisely, the dimension of $\mathcal{M}(X,n,0)$.

Another interpretation of $\dim H^0(X, \mathcal{E}^*)=n^2(g_X-1)+1$ comes from the moduli space of generalized parabolic bundle $\mathcal{M}_{par}(\widetilde{X},n,0)$, where $\widetilde{X}$ is normalization of $X$. By Theorem 1 in \cite{Bho1}, we know the dimension of $\mathcal{M}_{par}(\widetilde{X},n,0)$ is $n^2(g_{\widetilde{X}}-1)+1+n^2$, where the term $n^2$ is the dimension of the flag variety for the corresponding parabolic structure of $\widetilde{E}$. This flag variety is exactly the Grassmanian $Gr(n,2n)$, i.e. $n$-dimensional subspace of a $2n$-dimensional vector space. Note that $g_{X}=g_{\widetilde{X}}+1$. Thus we have
\begin{align*}
\dim \mathcal{M}_{par}(\widetilde{X},n,0)=n^2(g_{\widetilde{X}}-1)+1+n^2=n^2(g_X-1)+1=\dim \mathcal{M}(X,n,0).
\end{align*}
In fact, the above equality is not a coincidence. Proposition \ref{203} implies an one-to-one correspondence between generalized parabolic bundles and bundles over nodal curve. This correspondence is first discovered by Bhosle in \cite{Bho1}. Thus the dimension of the moduli spaces $\mathcal{M}_{par}(\widetilde{X},n,0)$ and $\mathcal{M}(X,n,0)$ are the same as expected.
\end{exmp}

\section{Deformation of Hitchin Pairs over a Nodal Curve}
In this section, we study the deformation of Hitchin pairs over a nodal curve $X$. We use two approaches to study this problem: one is generalizing Biswas and Ramanan's approach \cite{BisRam} to study the deformation of $L$-twisted Hitchin pairs over nodal curve; the second one is by using the correspondence between Hitchin pairs over nodal curve and generalized parabolic Hitchin pairs over its normalization, which is equivalent to study the deformation of the corresponding GPH over its normalization.

We want to remind the reader that Yokogawa studied the infinitesimal deformation theory for parabolic bundles \cite{Yoko2}. Together with Biswas and Ramanan's work, the deformation of parabolic Higgs bundles is studied in a similar way in \cite{GaGoMu}.  Note that the definition of the parabolic bundle is different from that of the generalized parabolic bundle. The usual parabolic structure depends on a fixed reduced effective divisor $D$ and involves a filtration over each point $x$ in the divisor $D$, while the generalized parabolic structure defines a filtration over each divisor $D_i$, $1 \leq i \leq s$, which can be a single point or the sum of points. In the case of nodal curve $X$, the divisor $D_i$ is the preimage of the node $x_i$ in the nomalization $\widetilde{X}$, which is the sum of two points. Although the definition of parabolic structure is slightly different, the approach to calculate deformations can be applied to the generalized parabolic Hitchin pair.

\subsection{First Approach}
With the same notation as in \S 3.1, let $C'$, $C$ be two local Artin rings satisfying the following exact sequence
\begin{align*}
0 \longrightarrow J \longrightarrow C' \longrightarrow C \longrightarrow 0.
\end{align*}
We can consider $J$ as a $k$-vector space, where $k$ is the residue field of $C$.
Let $X$ be a nodal curve over $C$ and let $X'$ be an extension of $X$ flat over $C'$. Note that $$X' \times_{\Spec C'} \Spec C=X.$$ We fix a line bundle $\mathbb{L}$ over $X$ together with its corresponding line bundle $\mathbb{L}'$ over $X'$. Let $(E,\Phi_{E})$ be a $\mathbb{L}$-twisted Hitchin pair over $X$. A \emph{deformation} $(E',\Phi')$ of $(E,\Phi_{E})$ is a $\mathbb{L}'$-twisted Hitchin pair over $X'$ such that its restriction to $X$ is $(E,\Phi_{E})$. Note that $\Phi_{E}$ can be considered as a section of $\text{End}(E) \otimes \mathbb{L}$.

Let us consider a special case. Let $C'=C[J]:=C \oplus J$. The algebra structure of $C'$ is given as follows: $$(m,n)(p,q)=(mn,mq+np).$$ Clearly, $J$ is a nilpotent ideal in $C'$. With the same notation as above, let $E'=E \times \Spec k[J]$. For a section $s$ of $\text{End}(E) \otimes J$, the corresponding automorphism of $E'$ is denoted by $1+s$. Moreover, if $v+w$ is a section of $\text{End}(E') \otimes \mathbb{L}'$, we have
\begin{align*}
\rho(1+s)(v+w)=v+w+\rho(s)(v),
\end{align*}
where $\rho$ is natural action of $\text{End}(E)$ on itself. The deformation complex $C_{J}^{\bullet}$ is defined as follows
\begin{align*}
C_{J}^{\bullet}: C_{J}^0 = \text{End}(E) \otimes J \xrightarrow{e(\Phi)} C_J^1 = \text{End}(E) \otimes \mathbb{L} \otimes J,
\end{align*}
where the map $e(\Phi)$ is given by
\begin{align*}
e(\Phi)(s)=-\rho(s)(\Phi).
\end{align*}

\begin{thm}\label{401}
Let $(E,\Phi)$ be a $\mathbb{L}$-twisted Hitchin pair over the nodal curve $X$. The set of deformations of $(E,\Phi)$ is isomorphic to $\mathbb{H}^1(C_J^{\bullet})$, where $C_J^{\bullet}$ is the complex $$C_J^{\bullet}:C_J^0 =End(E) \otimes J \xrightarrow{e(\Phi)} C_J^1 = \text{End}(E) \otimes \mathbb{L} \otimes J,$$
where $e(\Phi)$ is defined locally as above.
\end{thm}

\begin{proof}
The proof of this theorem is similar to that of Theorem 2.3 in \cite{BisRam}. We only give the construction of the deformation of $(E,\Phi)$ from an element in $\mathbb{H}^1(C_J^{\bullet})$.

Let $\mathcal{U}=\{U_i=\text{Spec}(A_i)\}$ be an open covering of $X$ by affine schemes. Set $$ \text{End}(E) \otimes J|_{U_i}=C^0_i, \quad \text{End}(E) \otimes \mathbb{L} \otimes J|_{U_i}=C^1_{i},$$
where $C^0_i$ and $C^1_i$ are $A_i$-modules. Similarly, modules $C^0_{ij}$ (resp. $C^1_{ij}$) are resctrictions of $C_J^0$ (resp. $C_J^1$) to $U_{ij}=U_i \bigcap U_j$. We consider the following \^Cech resolution of $C^{\bullet}_{J}$:
\begin{center}
\begin{tikzcd}
& 0 \arrow[d] & 0 \arrow[d] & \\
0 \arrow[r] & C_J^0 \arrow[r,"e(\Phi)"] \arrow[d,"d^0_0"] & C_J^1 \arrow[r] \arrow[d,"d^1_0"] & 0 \\
0 \arrow[r] & \sum C_i^0 \arrow[r,"e(\Phi)"] \arrow[d,"d^0_1"] & \sum C_i^1 \arrow[r] \arrow[d,"d^1_1"] & 0 \\
0 \arrow[r] & \sum C_{ij}^0 \arrow[r,"e(\Phi)"] \arrow[d,"d^0_2"] & C_{ij}^1 \arrow[r] \arrow[d,"d^1_2"] & 0 \\
& \vdots  & \vdots  &
\end{tikzcd}
\end{center}
The first hypercohomology group $\mathbb{H}^1(C^{\bullet}_J)$ can be calculated from the above diagram. Let $Z$ be the set of pairs $(s_{ij}, t_i)$, where $s_{ij} \in C^0_{ij}$ and $t_i \in C^1_i$ satisfying the following conditions:
\begin{enumerate}
\item $s_{ij}+s_{jk}=s_{ik}$ as elements of $C^0_{ijk}$.
\item $t_i-t_j=e(\Phi)(s_{ij})$ as elements of $C^1_{ij}$.
\end{enumerate}
Let $B$ be the subset of $Z$ consisting of elements $(s_i-s_j,e(\Phi)(s_i))$, where $s_i \in C^0_i$. The hypercohomology group $\mathbb{H}^1(C^{\bullet}_J)$ is $Z/B$.

Given an element $(s_{ij},t_i) \in Z$, we shall construct a $\mathbb{L}$-twisted Higgs bundle $(E',\Phi')$ on $X'$ such that $E'|_{X} \cong E$ and $\Phi' |_{X} \cong \Phi$.

For each $U_i[J]$, there is a natural projection $\pi: U_i[J] \rightarrow U_i$. Take the sheaf $E'_i=\pi^*(E|_{U_i})$. By the first condition of $Z$, we can identify the restrictions of $E'_i$ and $E'_j$ to $U_{ij}[J]$ by the isomorphism $1+s_{ij}$ of $E'_{ij}$. Therefore we get a well-defined quasi-coherent sheaf $E'$ on $X'$.

On each affine set $U_i[J]$, we have $\Phi_{i}+t_i : End(E'_i) \otimes \mathbb{L}$. It is easy to check $$e(\Phi_{i} + t_i)(1+s_{ij})=\Phi_{j}+t_j$$
by the second condition of $Z$. Therefore $\{\Phi_{i}+t_i\}$ can be glued together to give a global homomorphism $\Phi':E' \rightarrow  E' \otimes \mathbb{L}'$. Therefore, for each element in $Z$, we construct a deformation of $(E,\Phi)$.

Let $(s_{ij},t_i)$ be an element in $B$. In other words, $s_{ij}=s_i-s_j$ and $t_i=e(\Phi)(s_i)$. The identification of $E'_i \cong E'_j$ on $U_{ij}[J]$ is given by the isomorphism $$1+s_{ij}=1+(s_i-s_j).$$ Consider the following diagram
\begin{center}
\begin{tikzcd}
E'_{ij} \arrow[d,"1+s_{ij}"] \arrow[r,"1+s_{i}"] & E'_{ij} \arrow[d, "\text{Id}"]  \\
E'_{ij} \arrow[r,"1+s_{j}"] & E'_{ij}
\end{tikzcd}
\end{center}
The commutativity of the above diagram implies that $E'$ is trivial. Similarly, we have $$e(\Phi_{i} +t_i)(1+s_i)=\Phi_{i}.$$
Therefore the associated Hitchin pair $(E',\Phi')$ is isomorphic to $(\pi^*E,\pi^* \Phi)$. The above construction gives us a well-defined map from $\mathbb{H}^1(C_J^{\bullet})$ to the set of deformations of $(E,\Phi)$.

Note that given a deformation $(E',\Phi')$ of $(E,\Phi)$, we can define an element $(s_{ij},t_i)$ by restricting to the open set $U_i[J]$. It is easy to check that the element $(s_{ij},t_i)$ is a well-defined element in $\mathbb{H}^1(C_J^{\bullet})$. Thus we construct a map from the set of deformations of $(E,\Phi)$ to $\mathbb{H}^1(C_J^{\bullet})$.

It is easy to check that the above two maps are inverse to each other. Thus the set of deformations of $(E,\Phi)$ is isomorphic to $\mathbb{H}^1(C_J^{\bullet})$.

\end{proof}

\begin{rem}\label{402}
The above proof works for both a singular (nodal) curve and a smooth curve. It can be also applied to a general scheme $X$. More generally, the above proof can be generalized for an algebraic space or algebraic stack. Note that if we working on an algebraic space, the covering $\mathcal{U}=\{U_i=\text{Spec}(A_i)\}$ that we took in the proof should be an \'etale covering. Thus in the case of algebraic space or stack, the hypercohomology group we calculate is in fact the \'etale cohomology .
\end{rem}

\subsection{Second Approach}
By Theorem \ref{204}, we have a birational morphism between the moduli space $\mathcal{M}(X,n,d,\mathbb{L})$ and the moduli space $\mathcal{M}_{par}^{good}(\widetilde{X},n,d,\widetilde{\mathbb{L}})$ of good GPH, which is induced by the correspondence in Proposition \ref{203}. Thus studying the deformation theory of $\mathbb{L}$-twisted Hitchin pairs $(E,\Phi)$ over a nodal curve $X$ is equivalent to study the deformation theory of the corresponding $\widetilde{\mathbb{L}}$-twisted good GPH $(\widetilde{E},F(\widetilde{E}),\Phi_{\widetilde{E}})$ over $\widetilde{X}$.

Let $\text{ParEnd}(\widetilde{E})$ be the set of parabolic homomorphisms of the generalized parabolic bundle $\widetilde{E}$. As we discussed in \S2.2, we have the following exact sequence
\begin{align*}
0 \rightarrow \text{ParEnd}(\widetilde{E}) \rightarrow \text{End}(\widetilde{E}) \rightarrow \text{End}(E_D,E_D)/ P_D(E,E) \rightarrow 0.
\end{align*}
With respect to the notation in \S 4.1, the deformation complex $C_{par,J}^{\bullet}$ in the parabolic case is defined as follows
\begin{align*}
C_{par,J}^{\bullet}: C_{par,J}^0 = \text{ParEnd}(\widetilde{E}) \otimes J \xrightarrow{e(\Phi_{\widetilde{E}})} C_{par,J}^1 = \text{ParEnd}(\widetilde{E}) \otimes \widetilde{L} \otimes J.
\end{align*}

With the same proof as in Theorem \ref{401}, we have the following corollary.
\begin{cor}\label{403}
Let $(\widetilde{E},F(\widetilde{E}),\Phi_{\widetilde{E}})$ be a good generalized parabolic Higgs bundle over $\widetilde{X}$. The set of deformations of $(\widetilde{E},F(\widetilde{E}),\Phi_{\widetilde{E}})$ is isomorphic to $\mathbb{H}^1(C_{par,J}^{\bullet})$, where $C_{par,J}^{\bullet}$ is the complex $$C_{par,J}^{\bullet}: C_{par,J}^0 = \text{ParEnd}(\widetilde{E}) \otimes J \xrightarrow{e(\Phi_{\widetilde{E}})} C_{par,J}^1 = \text{ParEnd}(\widetilde{E}) \otimes \widetilde{\mathbb{L}} \otimes J.$$
\end{cor}

\begin{rem}\label{404}
Theorem \ref{401} and Corollary \ref{403} imply that $\mathbb{H}^1(C_{par,J}^{\bullet}(E)) \cong \mathbb{H}^1(C_{J}^{\bullet}(\widetilde{E}))$. This isomorphism can be understood from the exact sequence
\begin{align*}
0 \rightarrow \text{ParEnd}(\widetilde{E}) \rightarrow \text{End}(\widetilde{E}) \rightarrow \text{End}(E_D,E_D)/ P_D(E,E) \rightarrow 0.
\end{align*}
The space of parabolic homomorphisms $\text{ParEnd}(\widetilde{E})$ is exactly the homomorphisms $\text{End}(E)$ over the nodal curve. This property is implied in \cite[Section 1, 4]{Bho1}. Thus the complexes
\begin{align*}
& C_{par,J}^{\bullet}: C_{par,J}^0 = \text{ParEnd}(\widetilde{E}) \otimes J \xrightarrow{e(\Phi')} C_{par,J}^1 = \text{ParEnd}(\widetilde{E}) \otimes \widetilde{L} \otimes J \\
& C_{J}^{\bullet}: C_{J}^0 = \text{End}(E) \otimes J \xrightarrow{e(\Phi)} C_{J}^1 = \text{End}(E) \otimes L \otimes J,
\end{align*}
are isomorphic. Thus we have the isomorphism of the hypercohomology groups $\mathbb{H}^1(C_{par,J}^{\bullet}(E)) \cong \mathbb{H}^1(C_{J}^{\bullet}(\widetilde{E}))$.
\end{rem}

\bigskip

\noindent\small{\textsc{Department of Mathematics, Sun Yat-Sen University, China}}\\
\emph{E-mail address}:  \texttt{sunh66@mail.sysu.edu.cn}
\end{document}